\documentclass[a4paper]{amsart}

\usepackage[utf8]{inputenc}			%for utf8 characters in pdfLatex
\usepackage[T1]{fontenc}			%for accents encoding
\usepackage[british]{babel} 		%british english

\usepackage{
amsmath,			%\DeclareMathOperator, etc
amsthm,				%\newtheorem, etc
amssymb,			%amsfonts, \mathbb, \cap, \cup, etc
thmtools,			%Extra theorem enviroment options 
xfrac,				%lateral fractions \sfrac
faktor,				%quotient spaces \faktor
mathtools,			%\coloneqq, \mathclap, \prescript, etc
enumitem,			%enumeration options (alternative enumerate)
microtype,			%microtyping (eliminates most badbox errors)
hyperref,			%hyperlinks 
cleveref,			%for \cref
}

\newtheorem{theo}{Theorem}
\newtheorem{lem}[theo]{Lemma}
\newtheorem{coro}[theo]{Corollary}
\newtheorem{fact}[theo]{Fact}

\theoremstyle{definition}
\newtheorem{defi}[theo]{Definition}

\theoremstyle{remark}
\newtheorem{rmk}[theo]{Remark}

\crefname{theo}{Theorem}{Theorems}
\crefname{lem}{Lemma}{Lemmas}
\crefname{coro}{Corollary}{Corollaries}
\crefname{fact}{Fact}{Facts}
\crefname{defi}{Definition}{Definitions}
\crefname{rmk}{Remark}{Remarks}
\setenumerate{itemsep=0pt, topsep=3pt}

\newcommand*\N{\mathbb{N}}		% Natural numbers
\newcommand*\sth{:}				% such that 
\newcommand*\map{\colon}
\newcommand*\lang{{\mathtt{L}}}	 % language

\newcommand*\Comm{\mathrm{Comm}} % commensurator

\renewcommand\qedsymbol{\setbox0=\hbox{\ \ \ \footnotesize{\normalfont Q.E.D.}}\kern\wd0 \strut \hfill \kern-\wd0 \box0}

\begin{document}

\title{A short note on the Massicot-Wagner method}          

\author{Arturo Rodr\'\i guez Fanlo} 
\address{Autonomous University of Madrid, Ciudad Universitaria de Cantoblanco, 28049, Madrid, Spain}
\email{arturo.rodriguez@uam.es}
%\thanks{}

\date{\today}

\begin{abstract}
We provide a general abstract statement of the Massicot-Wagner method: our main result is an asymmetric version (i.e. a version for group actions) of the recursive Massicot-Wagner argument.  
\end{abstract}


\maketitle

\section*{Introduction}
The Massicot-Wagner method is a technique used in additive combinatorics and model theory to show that an approximate subgroup has a locally compact model (Lie model). This technique is now common practice in the area --- e.g., see 
\cite{%breuillard2012structure,
carolino2015structure,
%massicot2015approximate,
%dries2015approximate,
krupinski2019amenability,
conant2020smalltripling,
%conant2020approximate,		
%hrushovski2021orderamenability,
hrushovski2022amenability,
%krupinski2022genereting,
machado2023smalldoubling,
jing2023measuredoubling,
krupinski2024approximaterings,
machado2025closed,
conant2025compactifications,
rodriguez2025rough}.  

Locally compact models (Lie models) were introduced by Hrushovski in his seminal paper \cite{hrushovski2011stable} as a fundamental connection between approximate subgroups and compact neighbourhood of the identity of locally compact groups. The existence of Lie models for pseudo-finite approximate subgroups \cite[Theorem~4.2]{hrushovski2011stable} was the starting point for the complete classification of finite approximate subgroups by Breuillard, Green and Tao \cite{breuillard2012structure}. The proof of \cite[Theorem~4.2]{hrushovski2011stable} strongly relies on model-theoretic techniques from stability theory. 

Based on the work of Sanders \cite{sanders2010nonabelian} and Croot-Sisask \cite{croot2010probabilistic}, an alternative combinatorial proof of the Lie Model Theorem  that avoids model theory was presented in \cite{breuillard2012structure}. In \cite{massicot2015approximate}, Massicot and Wagner substantially improved this combinatorial proof. The Massicot-Wagner\footnote{We are facing Stigler's law of eponymy: historically, it should probably be called the Sanders-Croot-Sisask-Massicot-Wagner method (especially \cref{l:original basic massicot wagner}).} method is summarised in two fundamental results:

\begin{lem}[Basic Massicot-Wagner {\cite[Theorem~12]{massicot2015approximate}}] \label{l:original basic massicot wagner} Let $G$ be a group and let $\Lambda\subseteq G$ be an approximate subgroup. Let $(G,\mathcal{A},\mu)$ be a {content space}\footnote{See \cref{d:content}.} space invariant by left translations, with $\Lambda\in\mathcal{A}$ and $0<\mu(\Lambda)<\infty$. Then, for any $n\in\N$, there is an approximate subgroup $S\subseteq\Lambda^4$ commensurable with $\Lambda$ such that $S^n\subseteq \Lambda^4$.  

\emph{Model theoretic remark:} the set $S$ can be chosen definable in the language $(G,\cdot,\Lambda)$.
\end{lem}

\begin{theo}[Recursive Massicot-Wagner {\cite[Theorem~1]{massicot2015approximate}}] \label{t:original recursive massicot wagner} Let $G$ be a group and let $\Lambda\subseteq G$ be an approximate subgroup. Let $(G,\mathcal{A},\mu)$ be a content space invariant by left translations, with $\Lambda\in\mathcal{A}$ and $0<\mu(\Lambda)<\infty$. Then, $\Lambda^4$ has a locally compact model\footnote{See \cref{d:lie model}.}.

\emph{Model-theoretic remark:} the locally compact model of $\Lambda^4$ can be taken to be definable in the language $(G,\cdot,\Lambda)$ with countably many parameters.
\end{theo}

The proof of \cref{l:original basic massicot wagner} is purely combinatorial and dates back to Sanders \cite{sanders2010nonabelian} and Croot-Sisask \cite{croot2010probabilistic}. On the other hand, the proof of \cref{t:original recursive massicot wagner} relies on applying \cref{l:original basic massicot wagner} recursively to find a sequence of (definable) approximate subgroups $(D_i)_{i\in\N}$ contained in and commensurable with $\Lambda^4$ such that $D_{i+1}^2\subseteq D_i$ for all $i\in\N$. The proof of \cref{t:original recursive massicot wagner} then immediately concludes, as the existence of such a sequence is equivalent to the existence of a locally compact model. An elementary proof of this equivalence is provided in \cite[Theorem~3.5]{machado2025closed}. From the point of view of model theory, this equivalence is straightforward, as such a sequence is equivalent to the existence of a type-definable normal subgroup $N$ of $\langle \Lambda\rangle$ contained in $\Lambda^4$ with bounded index, and so equivalent to the existence of a locally compact model (obtained by taking the quotient by $N$ and using the logic topology).

\

In \cite{hrushovski2022amenability}, the Massicot-Wagner method was revisited. The authors enhanced the basic Massicot-Wagner argument in three key ways.
\begin{enumerate}[label={\rm (\roman*)}, wide]
\item First, they adapt the technique for use in ``asymmetric'' situations, i.e. group actions. 
\item Second, they note that the family of measurable sets only needs to be a bounded lattice of sets, i.e. it suffices to work with a finite mean space\footnote{See \cref{d:mean}.}.
\item Last, from a model-theoretic perspective, they avoid saturation and achieve semipositively definability\footnote{Definable using $\wedge$, $\vee$, $\exists$ and $\forall$, but no $\neg$.}. 
\end{enumerate}
\begin{lem}[Basic Massicot-Wagner {\cite[Proposition~2.11]{hrushovski2022amenability}}] \label{l:hkp basic massicot wagner} Let $G\curvearrowright E$ be a group acting on a set. Let $(G,\mathcal{A},\mu)$ be a finite mean space invariant by left translations, and let $(E,\mathcal{F},m)$ be a finite mean space invariant by the action of $G$. Let $A\in\mathcal{A}$ and $B\in\mathcal{F}$ satisfy $\mu(A)>0$ and $m(B)>0$. Assume that $A'B\in\mathcal{F}$ for any $A'$ which is a finite intersection of translates of $A$. Then, there is a symmetric subset $S\subseteq G$ such that finitely many (left) translates of $S$ cover $G$ and $S^n\subseteq S(AB)\coloneqq\{g\in G\sth gAB\cap AB\neq \emptyset\}$. 

\emph{Model-theoretic remark:} the set $S$ can be chosen to be semipositively definable in the language $(G,\cdot,G\setminus\Lambda,A)$.
\end{lem}

By contrast, the recursive Massicot-Wagner argument is not  substantially revisited in \cite{hrushovski2022amenability} --- it simply appears as a claim inside of the proof of \cite[Proposition~2.52]{hrushovski2022amenability}. Notably, the authors still consider the recursive Massicot-Wagner argument only in the ``symmetric'' case, i.e. when a group acts on itself. It is therefore natural to wonder whether an asymmetric version of the recursive Massicot-Wagner method can be proven in the same way as \cite[Proposition~2.52]{hrushovski2022amenability}. In this short note, we answer this question in the affirmative. Our main result is the following:

\begin{theo}[Recursive Massicot-Wagner] \label{t:recursive massicot wagner mean} Let $G\curvearrowright E$ be a group acting on a set and let $\Lambda\subseteq G$ be an approximate subgroup. Let $(G,\mathcal{A},\mu)$ be a mean space invariant by left translations and $(E,\mathcal{F},m)$ be a mean space invariant by the action of $G$. Let $A,C\in\mathcal{A}$ with $\Lambda A\subseteq C$ and $B\in\mathcal{F}$ satisfy $0<\mu(A)\leq\mu(C)<\infty$ and $0<m(B)\leq\mu(AB)<\infty$. Assume that $A'B\in\mathcal{F}$ for any $A'$ which is a finite intersection of translates of $A$. Suppose that $S(AB)\coloneqq\{g\in\langle\Lambda\rangle \sth m(gAB\cap AB)>0\}\subseteq \Lambda^n$. Then, $\Lambda^n$ has a locally compact model.

\emph{Model-theoretic remark:} the locally compact model of $\Lambda^n$ can be taken to be definable in the language $(G,\cdot,\Lambda,A)$ with countably many parameters.  
\end{theo}

\section{Preliminaries}
\begin{defi}[Content] \label{d:content} A \emph{content space} $(E,\mathcal{A},\mu)$ is a non-empty set $E$ with a ring $\mathcal{A}$ of subsets of $E$ and a finitely additive function $\mu\map \mathcal{A}\to[0,\infty]$ such that $\mu(\emptyset)=0$. In other words:
\begin{enumerate}[label={\rm (\roman*)}, wide]
\item $\mathcal{A}$ is a non-empty family of subsets of $E$ closed under finite unions, intersections and differences.
\item For $A,B\in\mathcal{A}$, with $A\cap B=\emptyset$, $\mu(A\cup B)=\mu(A)+\mu(B)$.
\item $\mu(\emptyset)=0$.
\end{enumerate}
We say that $(E,\mathcal{A},\mu)$ is a \emph{finite} content space if $E\in\mathcal{A}$ and $\mu(E)$ is finite.
\end{defi}
\begin{defi}[Mean]\label{d:mean} A \emph{mean space} $(E,\mathcal{A},\mu)$ is a non-empty set $E$ with a lattice $\mathcal{A}$ of subsets of $E$ containing $\emptyset$ and a monotone additive function $\mu\map \mathcal{A}\to [0,\infty]$ such that $\mu(\emptyset)=0$. In other words:
\begin{enumerate}[label={\rm (\roman*)}, wide]
\item $\mathcal{A}$ is a family of subsets of $E$ closed under finite unions and intersections with $\emptyset\in\mathcal{A}$.
\item For $A,B\in\mathcal{A}$, $\mu(A)\leq\mu(B)$ if $A\subseteq B$.
\item For $A,B\in\mathcal{A}$, $\mu(A\cup B)=\mu(A)+\mu(B)-\mu(A\cap B)$ if $\mu(A\cap B)<\infty$. 
\item $\mu(\emptyset)=\emptyset$.
\end{enumerate}
We say that $(E,\mathcal{A},\mu)$ is a \emph{finite} mean space if $E\in\mathcal{A}$ and $\mu(E)$ is finite.
\end{defi}
%\begin{rmk} Let $(E,\mathcal{A},\mu)$ be a mean space. By induction on $n$, we get that 
%\[\sum_i\mu(A_i)\geq \mu(\bigcup A_i)\geq \sum_i \mu(A_i)-\sum_{i<j}\mu(A_i\cap A_j)\]
%for any $A_1,\ldots,A_n\in\mathcal{A}$ such that $\mu(A_i\cap A_j)<\infty$ for any $i\neq j$. 
%\begin{proof} By induction on $n\geq 2$. It is trivial for $n=2$. Suppose it holds for $n$. Consider $A_0,\ldots,A_n\in \mathcal{A}$. On the one hand, we get
%\[\mu(\bigcup_{i\leq n} A_i)=\mu(\bigcup_{i<n} A_i)+\mu(A_n)-\mu(A_n\cap \bigcup_{i<n} A_i)\leq \sum_{i\leq n}\mu(A_i).\]
%On the other hand, 
%\[\mu(\bigcup A_i)=\mu(\bigcup_{i<n} A_i)+\mu(A_n)-\mu(A_n\cap \bigcup_{i<n} A_i)\geq \sum_{i\leq n} \mu(A_i)-\sum_{i<j<n}\mu(A_i\cap A_j)-\sum_{i<n} \mu(A_i\cap A_n)\]  
%\end{proof}
%\end{rmk}

\begin{defi}[Approximate subgroups]\label{d:approximate} Let $G$ be a group and $A,B\subseteq G$. 
\begin{enumerate}[label={\rm{(\arabic*)}}, wide]
\item $A$ and $B$ are \emph{(left) commensurable} if finitely many left translates of $A$ cover $B$ and finitely many left translates of $B$ cover $A$. More precisely, we say that $A$ and $B$ are \emph{$k$-commensurable} if $A\subseteq \Delta B$ and $B\subseteq \Delta A$ for some $\Delta\subseteq G$ with $|\Delta|=k$.
\item $A$ is a \emph{$k$-approximate subgroup} if it is symmetric (i.e. $A^{-1}=A$) contains the identity and is $k$-commensurable with its set of pairwise products $A^2\coloneqq A\cdot A=\{ab\sth a,b\in A\}$.
\end{enumerate}  
\end{defi}
\begin{rmk} \label{r:index intrinsic} If finitely many translates of $A$ cover $B$, there is $\Delta\subseteq AB^{-1}$ such that $B\subseteq \Delta A$.
\end{rmk}

\begin{defi}[Commensurator] The \emph{commensurator} of $X$ in $G$ is the set
\[\Comm_G(X)\coloneqq \{g\in G\sth g^{-1}Xg\text{ and }X\text{ are commensurable}\}.\] 
\end{defi}

\begin{fact}[{\cite[Lemma~5.1]{hrushovski2022beyond}}] \label{f:commensurator} Let $G$ be a group and $\Lambda\subseteq G$ be an approximate subgroup. Then, $\Comm_G(\Lambda)$ is a subgroup of $G$ containing $\Lambda$ and $\Comm_G(\Lambda)=\bigcup \{X\subseteq G \sth X\text{ approximate subgroup commensurable with }\Lambda\}$. 
\end{fact}
We will use the following lemma several times:
\begin{lem}\label{l:symmetric sets of the commensurator} Let $G$ be a group, $\Lambda\subseteq G$ an approximate subgroup and $X\subseteq G$ a symmetric subset containing the identity commensurable with $\Lambda$. Then, $X$ is an approximate subgroup if and only if $X\subseteq \Comm(\Lambda)$.
\begin{proof} By \cref{f:commensurator}, if $X$ is an approximate subgroup, then $X\subseteq \Comm(\Lambda)$. Conversely, if $X\subseteq \Comm(\Lambda)$ is symmetric and commensurable with $\Lambda$, then there is $\Delta\subseteq \Comm(\Lambda)$ finite such that $X\subseteq \Delta \Lambda$ by \cref{r:index intrinsic}. Thus, by symmetry, $X^2\subseteq \Delta \Lambda^2\Delta^{-1}$, so $X^2$ is covered by finitely many left translates of $\Lambda\Delta^{-1}$. Now, for each $g\in\Delta$, $g^{-1}\in\Comm(\Lambda)$, so $\Lambda g^{-1}$ is covered by finitely many left translates of $\Lambda$. As $\Delta$ is finite, we conclude that $\Lambda\Delta^{-1}$ is covered by finitely many translates of $\Lambda$, so $X^2$ is covered by finitely many translates of $X$.
\end{proof}
\end{lem}

\begin{defi}[Thickness]\label{d:thick} Let $G$ be a group and $A,B\subseteq G$. We say that $A$ is \emph{(left) $k$-thick} in a set $B$ if, for any $b_1,\ldots,b_{k+1}\in B$, there are $i\neq j$ such that $b_i^{-1}b_j\in A$. 
\end{defi}
\begin{rmk} \label{r:thick and commensurable} If $A$ is $k$-thick in $B$, then $B\subseteq \Delta A$ for some $\Delta\subseteq B$ with $|\Delta|\leq k$. On the other hand, if $k$ translates of $A$ cover $B$, then $A^{-1}A$ is $k$-thick in $B$.
\end{rmk}
\begin{rmk} \label{r:hereditary thick} If $A$ is $k$-thick in $B$ and $B_0\subseteq B$, then $A\cap B_0^{-1}B_0$ is $k$-thick in $B_0$.  
\end{rmk}

\begin{defi}[Lie models]\label{d:lie model} Let $G$ be a group, $\Lambda\subseteq G$ an approximate subgroup and $H\leq G$ with $\Lambda\subseteq H$. A \emph{locally compact model} for $\Lambda$ with domain $H$ is a group homomorphism $f\map H\to L$ to a locally compact group $L$ such that 
\begin{enumerate}[label={\rm{(\roman*)}}, wide]
\item $f(\Lambda)$ is contained in a compact set;
\item there is a neighbourhood of the identity $U$ such that $f^{-1}(U)\subseteq \Lambda$.
\end{enumerate}
We say that $f$ is a \emph{Lie model} if $L$ is a Lie group. 
\end{defi}

\begin{fact}[{\cite[Theorem~3.5]{machado2025closed}}] Let $G$ be a group, $\Lambda$ an approximate subgroup of $G$ and $H\leq G$ with $\Lambda\subseteq H$. Then, $\Lambda$ has a locally compact model with domain $H$ if and only if $H\subseteq \Comm(\Lambda)$ and there is a decreasing sequence $(D_i)_{i\in\N}$ of approximate subgroups contained in and commensurable with $\Lambda$ such that $D_{i+1}^2\subseteq D_i$.
\end{fact}

\section{Massicot-Wagner systems}
Let $G$ be a group and $\Lambda,\Gamma\subseteq G$ subsets. A \emph{Massicot-Wagner system} in $\Lambda$ relative to $\Gamma$ is a sequence $\ell=(\ell_k)_{k\in\N}$ of families of non-empty subsets of $G$ with the following property:
\begin{enumerate}[wide]
\item[$\boldsymbol\ast_{\mathrm{g}}$] If $W\in \ell_k$, then $\Lambda$ is covered by finitely many translates of $\{g\in \Gamma\sth gW\cap W\in \ell_{k-1}\text{ or }g^{-1}W\cap W\in\ell_{k-1}\}$.  
\end{enumerate}

A \emph{thick Massicot-Wagner system} in $\Lambda$ relative to $\Gamma$ is a sequence $(\ell_k)_{k\in\N}$ of families of non-empty subsets of $G$ with the following property:
\begin{enumerate}[wide]
\item[$\boldsymbol\ast_{\mathrm{t}}$] If $W\in \ell_k$, then $\{g\in \Gamma\sth gW\cap W\in \ell_{k-1}\text{ or }g^{-1}W\cap W\in \ell_{k-1}\}$ is thick in $\Lambda$.  
\end{enumerate}
\begin{rmk} By \cref{r:thick and commensurable}, every thick Massicot-Wagner system is a Massicot-Wagner system.\end{rmk}
\begin{rmk} By \cref{r:hereditary thick}, if $\ell$ is a thick Massicot-Wagner system in $\Lambda$ relative to $\Gamma$, then $\ell$ is a thick Massicot-Wagner system in $\Lambda_0$ relative to $\Gamma\cap \Lambda_0^{-1}\Lambda_0$ for any $\Lambda_0\subseteq \Lambda$.
\end{rmk} 
\begin{rmk} A key innovation in our approach to the Massicot-Wagner method is the introduction of $\Gamma$. This seemingly minor technical modification is conceptually crucial for proving \cref{t:recursive massicot wagner mean}. In the usual applications of the Massicot-Wagner method to date, $\Gamma=G$ or $\Gamma=\langle\Lambda\rangle$.\end{rmk}

Massicot-Wagner systems aim to capture different levels of ``largeness''. Two basic examples were explicitly introduced in \cite{hrushovski2022amenability} and have been implicitly used in most of the applications of the Massicot-Wagner method to date.

\begin{defi} Let $G$ be a group and $\Lambda,\Gamma\subseteq G$. 
\begin{enumerate}[label={\rm(\arabic*)}, wide]
\item The \emph{largest Massicot-Wagner system} $\ell^{\mathrm{g}}\coloneqq\ell^{\mathrm{g}}(\Lambda;\Gamma)$ in $\Lambda$ relative to $\Gamma$ is recursively defined by
\[\begin{array}{ccl} W\in \ell^{\mathrm{g}}_0&\Leftrightarrow &W\neq \emptyset;\\ W\in \ell^{\mathrm{g}}_k&\Leftrightarrow &\Lambda\text{ is coverd by finitely many translates of }\\ && \{g\in \Gamma\sth gW\cap W\in \ell^{\mathrm{g}}_{k-1}\text{ or }g^{-1}W\cap W\in \ell^{\mathrm{g}}_{k-1}\}.\end{array}\]
\item The \emph{largest thick Massicot-Wagner system} $\ell^{\mathrm{t}}\coloneqq\ell^{\mathrm{t}}(\Lambda;\Gamma)$ in $\Lambda$ relative to $\Gamma$ is recursively defined by
\[\begin{array}{ccl} W\in \ell^{\mathrm{t}}_0&\Leftrightarrow &W\neq \emptyset;\\ W\in \ell^{\mathrm{t}}_k&\Leftrightarrow &\{g\in \Gamma\sth gW\cap W\in \ell^{\mathrm{t}}_{k-1}\text{ or }gW\cap W\in \ell^{\mathrm{t}}_{k-1}\}\text{ is thick in }\Lambda .\end{array}\]
\end{enumerate}
\end{defi} 

One of the main advantages of the Massicot-Wagner systems $\ell^{\mathrm{g}}$ and $\ell^{\mathrm{t}}$ is that they are canonical, i.e. invariant by group automorphisms leaving invariant $\Lambda$ and $\Gamma$. Additionally, from the point of view of model theory, they have the advantage of being $\bigvee$-definable with respect to $(G,\cdot,\Gamma,\Lambda)$ (assuming enough saturation). We leave the discussion about definability to the end of the paper. 

Another typical Massicot-Wagner system comes from mean spaces. Let $G$ be a group and $(G,\mathcal{A},\mu)$ a mean space invariant by left translations. The \emph{Massicot-Wagner system $\ell^\mu$ associated to $\mu$} is defined by 
\[\ell^\mu_k=\{W\in\mathcal{A}\sth \mu(W)>0\}\text{ for all }k\in\N.\] 

\begin{lem} \label{l:massicot wagner systems from means} Let $G$ be a group, $\Lambda\subseteq G$ and $(G,\mathcal{A},\mu)$ a mean space invariant by left translations. Let $A,C\in\mathcal{A}$ with $\Lambda A\subseteq C$ satisfy $0<\mu(A)\leq\mu(C)<\infty$. Then, for all $W\subseteq A$ with $\mu(W)>0$, the set $S_\mu\coloneqq\{g\in \Lambda^{-1}\Lambda\sth \mu(gW\cap W)>0\}$ is $\lfloor\mu(B)/\mu(W)\rfloor$-thick in $\Lambda$. In particular, $\ell^\mu$ restricted to subsets of $A$ is a thick Massicot-Wagner system.  
\end{lem}
This is a well known fact. We provide the proof for the sake of completeness.
\begin{proof} Pick $\Delta\subseteq \Lambda$ finite such that $g^{-1}h\notin S_\mu$ for $g,h\in\Delta$ with $g\neq h$. Then, $\mu(gW\cap hW)=0$ for any $g,h\in \Delta$ with $g\neq h$. Consequently, 
\[\mu(B)\geq \mu(\Delta W)=|\Delta|\cdot \mu(W),\]
so $|\Delta|\leq \mu(B)/\mu(W)$ is finite. Thus, $S_\mu$ is $\lfloor\mu(B)/\mu(W)\rfloor$-thick in $\Lambda$.
\end{proof}
\begin{rmk} In fact, $\{g\in\Lambda^{-1}\Lambda\sth \mu(gW\cap W)\geq \sfrac{\mu(W)^2}{2\mu(B)}\}$ is $\lceil 2\mu(B)/\mu(W)\rceil$-thick in $\Lambda$. This remark can be used to achieve greater quantitative control in applications, for instance by using more quantitative versions of $\ell^\mu$. 
%\begin{proof} Write $t=\lceil 2\mu(B)/\mu(W)\rceil$, so $\mu(W)>\sfrac{2\mu(B)}{t}$. Write $S=\{g\in\Lambda^{-1}\Lambda\sth \mu(gW\cap W)\geq \sfrac{\mu(W)^2}{2\mu(B)}\}$. Suppose there is $\Delta\subseteq \Lambda$ with $|\Delta|=t+1$ such that $g^{-1}h\notin S$ for $g,h\in\Delta$ with $g\neq h$. Then,
%\[\mu(B)\geq \mu(\Delta W)>(t+1)\mu(W)-\frac{t(t+1)}{2}\frac{\mu(W)^2}{2\mu(B)}\geq \frac{t+1}{t}2\mu(B)-\frac{t(t+1)}{2}\frac{4\mu(B)^2}{2t^2\mu(B)}=2\frac{t+1}{t}\mu(B)-\frac{t+1}{t}\mu(B)=\frac{t+1}{t}\mu(B),\]
%a contradiction. 
%\end{proof}
\end{rmk}
In our treatment of the Massicot-Wagner method, we have decided to relax the hypothesis on the $\ell$ systems as much as possible (for example, we are not assuming that they are invariant by translations). We have taken this approach based on our belief that the ability to choose $\ell$ is one of the primary strengths of the method.

%\begin{lem} Let $G$ be a group, $\Lambda\subseteq G$ and $(G,\mathcal{A},\mu)$ a mean space invariant by left translations. Let $A\in\mathcal{A}$ such that $0<\mu(A)$ and $\mu(C)<\infty$ for some $C\in\mathcal{A}$ with $\Lambda A\subseteq C$. Let $(\ell_k)_{k\in\N}$ be a sequence of families of non-empty subset of $G$ satisfying:
%\begin{itemize}
%\item If $\mu(W)>0$, then $W\in \ell_0$. 
%\item If $\{g\in \Lambda^{-1}\Lambda\sth gW\cap W\in \ell_{k-1}\}$ is thick in $\Lambda$, then $W\in \ell_k$.  
%\end{itemize}
%Then, $A\in\bigcap \ell_k$. In particular, $A\in\bigcap \ell^{\mathrm{t}}_k$.
%\begin{proof} We prove by induction on $k$ that, if $W\in\mathcal{A}$ with $W\subseteq A$ satisfies $\mu(W)>0$, then $W\in \ell_k$. Write $S_\mu=\{g\in \Lambda^{-1}\Lambda\sth \mu(gW\cap W)>0\}$ and $S_{\ell_{k-1}}=\{g\in\Lambda^{-1}\Lambda\sth gW\cap W\in\ell_{k-1}\}$. By induction hypothesis, $S_\mu\subseteq S_{\ell_{k-1}}$.  Since $S_\mu$ is $\lfloor\mu(B)/\mu(W)\rfloor$-thick in $\Lambda$ by \cref{l:thick mean stabilizer}, we get that $S_{\ell_{k-1}}$ is also thick in $\Lambda$, so $W\in\ell_k$. 
%\end{proof}
%\end{lem}

\section{Massicot-Wagner method} 
We start by recalling the following fact; see \cite[Lemma~2.10]{hrushovski2022amenability}.
\begin{fact}\label{f:lemma basic massicot wagner} Let $G\curvearrowright E$ be a group acting on a set and let $(E,\mathcal{F},m)$ be a mean space invariant by the action of $G$. Let $Z\in \mathcal{F}$, $g_1,g_2\in G$ and $\varepsilon_1,\varepsilon_2\in \mathbb{R}_{>0}$. Suppose $m(g_1Z\cap Z)>(1-\varepsilon_1)\cdot m(Z)$ and $m(g_2Z\cap Z)>(1-\varepsilon_2)\cdot m(Z)$ with $m(Z)<\infty$. Then, $m(g_1g_1Z\cap Z)>(1-\varepsilon_1-\varepsilon_2)\cdot m(Z)$. 
\end{fact}
%We provide the proof for the sake of completeness.
%\begin{proof} Since $m(Z)<\infty$, we get that all measures involved are finite by monotonicity. Note that $m(gZ\cap Z)>(1-t)\cdot m(Z)$ if and only if $m(Z)-m(gZ\cap Z)<t\cdot m(Z)$. As $m(g_1Z\cap Z)>(1-\varepsilon_1)\cdot m(Z)$ and $m(g_2Z\cap Z)>(1-\varepsilon_2)\cdot m(Z)$, we get
%\[2\cdot m(Z)-m(g_1Z\cap Z)-m(g_2Z\cap Z)<(\varepsilon_1+\varepsilon_2)\cdot m(Z).\]
%By monotonicity, $m(g_1Z\cup (g_1g_2Z\cap Z))\geq m(g_1Z\cap (g_1g_2Z\cup Z))$. Now, $m(g_1Z\cup (g_1g_2Z\cap Z))=m(g_1Z)+m(g_1g_2Z\cap Z)-m(g_1g_2Z\cap g_1Z\cap Z)$ and $m(g_1Z\cap (g_1g_2Z\cup Z))=m(g_1Z\cap g_1g_2Z)+m(g_1Z\cap Z)-m(g_1g_2Z\cap g_1Z\cap Z)$. Therefore, \[m(g_1Z)+ m(g_1g_2Z\cap Z)\geq m(g_1Z\cap g_1g_2Z)+m(g_1Z\cap Z).\]
%By invariance,
%\[m(Z)+m(g_1g_2Z\cap Z)\geq m(Z\cap g_2Z)+m(g_1Z\cap Z),\]
%so
%\[m(Z)+m(Z)-m(Z\cap g_2Z)-m(g_1Z\cap Z)\geq m(Z)-m(g_1g_2Z\cap Z),\]
%concluding $m(Z)-m(g_1g_2Z\cap Z)<(\varepsilon_1+\varepsilon_2)\cdot \mu(Z)$. Hence, $m(g_1g_2Z\cap Z)>(1-\varepsilon_1-\varepsilon_2)\cdot m(Z)$.
%\end{proof}
We now reprove \cite[Proposition~2.11]{hrushovski2022amenability}. In doing so, we introduce an additional set $\Gamma$ to the statement which will be crucial for the recursive argument. 
\begin{lem}[Basic Massicot-Wagner]\label{l:basic massicot wagner} Let $G\curvearrowright E$ be a group acting on a set. Let $\Lambda\subseteq G$ and $\Gamma\subseteq G$ with $\Gamma$ symmetric containing the identity. Let $(E,\mathcal{F},m)$ be a mean space invariant by the action of $G$. Let $\ell=(\ell_k)_{k\in\N}$ be a Massicot-Wagner system in $\Lambda$ relative to $\Gamma$. Let $A\in \bigcap \ell_k$ and $B\in \mathcal{F}$ satisfy: 
\begin{enumerate}[label={\rm (\alph*)}, wide]
\item $A'B\in \mathcal{F}$ for any $A'\subseteq A$ which is a finite intersection of translates of $A$. 
\item $0<m(B)\leq m(AB)<\infty$. 
\end{enumerate}
Set $S_\Gamma(AB)\coloneqq\{g\in \langle \Gamma\rangle\sth m(gAB\cap AB)>0\}$. Then, for any $n\in\N$, there is a symmetric subset containing the identity $D\subseteq \Gamma\cap AA^{-1}$ such that finitely many translates of $D$ cover $\Lambda$ and $D^n\subseteq S_\Gamma(AB)$.
\end{lem}
The proof is essentially identical to \cite[Proposition~2.11]{hrushovski2022amenability}. We provide it for the sake of completeness and to remark the role played by $\Gamma$.
\begin{proof} Set $\varepsilon=\sfrac{1}{n}$ and $\lambda=\sqrt{1+\varepsilon}>1$. Let $\mathcal{A}$ be the family of subsets of $A$ that are finite intersections of translates of $A$. Let $f(k)=\inf \{m(A'B)\sth A'\in \mathcal{A}\cap\ell_k\}$. By assumption, $A\in \mathcal{A}\cap \ell_k$ for all $k$, so $f(k)$ is well defined and $0<m(B)\leq f(k)\leq m(AB)<\infty$. Consequently, we cannot have $f(k)\geq \lambda\, f(k-1)$ for all $k$. Therefore, there is $k$ such that $f(k)<\lambda\, f(k-1)$. In fact, we can take $k\leq \lfloor\log_\lambda(m(AB)/m(B))\rfloor$. Take $W\in \mathcal{A}\cap \ell_k$ satisfying $m(WB)<\lambda\, f(k)$. Note that $W\subseteq A$. 

Let $D=\{g\in \Gamma \sth gW\cap W\in \ell_{k-1}\text{ or }g^{-1}W\cap W\in\ell_{k-1}\}\cup \{1_G\}$. For any $g\in D$, $gW\cap W\neq \emptyset$, so $g\in WW^{-1}\subseteq AA^{-1}$. Hence, $D\subseteq \Gamma\cap AA^{-1}$. Since $W\in\ell_k$, finitely many translates of $D$ cover $\Lambda$. For $g\in D$ with $g\neq 1_G$, we have $gW\cap W\in\ell_{k-1}$ or $g^{-1}W\cap W\in\ell_{k-1}$, so $f(k-1)\leq m((gW\cap W)B)$ or $f(k-1)\leq m((g^{-1}W\cap W)B)$. By invariance of $m$, we have $f(k-1)\leq m((gW\cap W)B)$. Therefore,
\[\begin{aligned}
m(gWB\cap WB)&\geq m((gW\cap W)B)\geq f(k-1)\\
&>f(k)/\sqrt{1+\varepsilon}>m(WB)/(1+\varepsilon)\\
&>(1-\varepsilon)\, m(WB).\end{aligned}\]  
By \cref{f:lemma basic massicot wagner}, we get $m(gAB\cap AB)\geq m(g WB\cap WB)>(1-n\varepsilon)\, m(WB)>0$ for all $g\in D^n$. Consequently, $g\in S_\Gamma(AB)$.  
\end{proof}
\begin{rmk} Actually, the proof gives more information about $D$. We know that $D=\{g\in\Gamma\sth gW\cap W\in\ell_{k-1}\text{ or }g^{-1}W\cap W\in\ell_{k-1}\}\cup\{1_G\}$ for some $W\subseteq A$ which is a finite intersection of translates of $A$ with $W\in\ell_k$ and $k\leq \lfloor\log_\lambda(m(AB)/m(B))\rfloor$. Depending on the definition of $\ell$, this extra information may be relevant.  
\end{rmk}

To obtain a locally compact model, the basic Massicot-Wagner argument should be applied iteratively to find a decreasing sequence $(D_i)_{i\in \N}$ with $D_{i+1}^2\subseteq D_i$ for $i\in\N$. This recursive argument requires control over $S(AB)$. Generally, we only know that $S(AB)\subseteq A\cdot S(B)\cdot A^{-1}$ with $S(B)=\{g\sth gB\cap B\neq \emptyset\}$, meaning that the recursive argument usually involves modifying both $A$ and $B$ at each stage. While this strategy is relatively straightforward when $G$ acts on itself and $\Lambda=A=B$, it is completely unclear how to change $A$ and $B$ for arbitrary group actions. Our solution is to keep $A$ and $B$ fixed and use $\Gamma$ to control $S_\Gamma(AB)$.

\begin{theo}[Recursive Massicot-Wagner] \label{t:recursive massicot wagner} Let $G\curvearrowright E$ be a group acting on a set and let $\Lambda$ be an approximate subgroup of $G$. Let $(E,\mathcal{F},m)$ be a mean space invariant by the action of $G$. Let $(\ell_k)_{k\in\N}$ be a thick Massicot-Wagner system in $\Lambda$ (relative to $\langle\Lambda\rangle$). Let $A\in \bigcap \ell_k$, $B\in \mathcal{F}$ and $n\in\N$ satisfy:
\begin{enumerate}[label={\rm (\alph*)}, wide]
\item $A'B\in \mathcal{F}$ for any $A'\subseteq A$ which is a finite intersection of translates of $A$. 
\item $0<m(B)\leq m(AB)<\infty$.
\item $S(AB)\coloneqq\{g\in \langle \Lambda\rangle\sth m(gAB\cap AB)>0\}\subseteq \Lambda^n$.
\end{enumerate}
Then, $\Lambda^n$ has a locally compact model.
\end{theo}
\begin{proof} By recursion, we define a sequence of numbers $(k_i)_{i\in\N}$ in $\N_{>0}$ and a sequence $(D_i)_{i\in\N}$ of commensurable symmetric subsets containing the identity such that $D^{2k_i}_{i+1}\subseteq D^{k_i}_i$. Set $D_0=\Lambda$. Suppose $D_i$ is given. Since $D_i$ is commensurable to $D_0=\Lambda$, we find that finitely many translates of $D_i$ cover $S_{D_i}(AB)=\{g\in\langle D_i\rangle \sth m(gAB\cap AB)>0\}$. Therefore, there is $k_i$ such that $S_{D_i}(AB)\subseteq D_i^{k_i}$ (in particular, we take $k_0=n$). By \cref{r:hereditary thick}, we note that $\ell$ is a thick Massicot-Wagner system in $D_i$ (relative to $\langle D_i\rangle$). By \cref{l:basic massicot wagner} applied to $D_i$, there is a symmetric subset $D_{i+1}\subseteq \langle D_i\rangle$ containing the identity such that finitely many translates of $D_{i+1}$ cover $D_i$ and $D_{i+1}^{2k_i}\subseteq S_{D_i}(AB)\subseteq D_i^{k_i}$. Now, consider the decreasing sequence $(D^{k_i}_i)_{i\in\N}$ of subsets of $\Lambda^n$. Then, this sequence satisfies the conditions of \cite[Theorem~3.5]{machado2025closed}. Consequently, $\Lambda^n$ has a locally compact model. 
\end{proof}
\begin{proof}[Proof of \cref{t:recursive massicot wagner mean}] It is a particular case of \cref{t:recursive massicot wagner}, applying \cref{l:massicot wagner systems from means}.
\end{proof}

\

\section{Definable versions}
For this section, fix a language $\lang$ and an $\lang$-structure $M$. 
\begin{enumerate}[label={}, wide]
\item By \emph{definable} we mean definable with parameters. For a subset $A$, \emph{$A$-definable} means definable with parameters in $A$. By $0$-definable we mean definable without parameters.
\item By \emph{$\bigvee_{<\kappa}$-definable} we mean a union of less than $\kappa$ definable sets. For a subset $A$, by \emph{$\bigvee_A$-definable} we mean a union of $A$-definable sets.
\item By \emph{$\bigwedge_{<\kappa}$-definable} we mean a conjunction of less than $\kappa$ definable sets, i.e. \emph{type-definable} by a partial type (with parameters) of size at most $\kappa$. For a subset $A$, by \emph{$\bigwedge_A$-definable} we mean a conjunction of $A$-definable sets, i.e. \emph{type-definable} by a partial type with parameters in $A$.
\item By \emph{semipositively definable} we mean definable by a formula using $\wedge,\vee,\exists,\forall$ but no $\neg$. By \emph{semipositively $\bigvee$-definable}, we mean a $\bigvee$-definable set which is a union of semipositively definable sets.
\item By the \emph{$A$-logic topology} we mean the topology whose closed sets are the $\bigwedge_A$-definable subsets.
\item Let $G$ be a definable group and $\Lambda\subseteq G$ a definable approximate subgroup. Let $H\leq G$ with $\Lambda\subseteq H$. Let $f\map H\to L$ be a locally compact model of $\Lambda$ with domain $H$. We say that $f$ is \emph{$A$-definable} if $H$ is $\bigvee_A$-definable and $f$ is continuous from the $A$-logic topology.
\end{enumerate}
\begin{rmk} The definition of definable locally compact models (Lie models) we have introduced agrees with the typical notion used in model theory. The typical definition in model theory is a pair $(K,H)$ with $K\triangleleft H\leq G$ where $K$ is $\bigwedge_A$-definable contained in $\Lambda$, $H$ is $\bigvee_A$-definable and $K$ has bounded index in $H$. Indeed, for such a pair, the quotient map $\pi\map H\to \faktor{H^*}{K^*}$ is a locally compact model of $\Lambda$, where $H^*/K^*$ is computed in an enough saturated elementary extension and with the global logic topology. We refer to \cite{rodriguez2023piecewise} for more details.

On the other hand, suppose $f\map H\to L$ is $A$-definable. Consider a sufficiently saturated elementary extension $M^*$, and let $G^*$, $\Lambda^*$ and $H^*$ be the corresponding elementary extensions of $G$, $\Lambda^*$ and $H^*$. Note that $G\subseteq G^*$ is dense in the $A$-logic topology. Consequently, there is a unique continuous extension $f^*\map H^*\to L$ of $f$. By continuity, $f^*$ is a group homomorphism. Since $f^{-1}(U)\subseteq \Lambda$, we get that ${f^*}^{-1}(U)\subseteq \Lambda^*$ by continuity. Also, $f^*(\Lambda^*)\subseteq \overline{f(\Lambda)}$ by continuity as $\Lambda$ is dense in $\Lambda^*$. Therefore, $f^*\map H^*\to L$ is a locally compact model too. For any compact set $C\subseteq L$, we have that ${f^*}^{-1}(C)\subseteq \Delta\Lambda^*$ for some finite set $\Delta$, so ${f^*}^{-1}(C)$ is $\bigwedge_A$-definable by continuity. Consequently, $f^*$ is a proper map. Since $L$ is locally compact Hausdorff, by \cite[Theorem~3.7.2]{engelking1989general}, $f^*$ is also a closed map. By the closed isomorphism theorem for topological groups \cite[Remark~2.1(4)]{rodriguez2023piecewise}, $f^*$ factors as an isomorphism $\faktor{H^*}{K^*}\cong L$ with $K^*=\ker\, f^*$ and $\faktor{H^*}{K^*}$ with the $A$-logic topology. Since $L$ is Hausdorff, we get that $K^*$ has bounded index in $H^*$. Thus, the restriction $(K,H)$ to $M$ is a locally compact model (Lie model) in the usual sense of model theory.
\end{rmk}

Let $G$ be a definable group and $\Lambda,\Gamma\subseteq G$. A \emph{definable Massicot-Wagner system} in $\Lambda$ relative to $\Gamma$ is a sequence $\ell=(\ell_k)_{k\in\N}$ of families of non-empty definable subsets of $G$ with the following property:
\begin{enumerate}[wide]
\item[$\boldsymbol\ast_{\mathrm{dg}}$] If $W\in \ell_k$, then $\Lambda$ is covered by finitely many translates of some definable subset of $\{g\in \Gamma\sth gW\cap W\in \ell_{k-1}\text{ or }g^{-1}W\cap W\in\ell_{k-1}\}$.  
\end{enumerate}

A \emph{definable thick Massicot-Wagner system} in $\Lambda$ relative to $\Gamma$ is a sequence $(\ell_k)_{k\in\N}$ of families of non-empty definable subsets of $G$ with the following property:
\begin{enumerate}[wide]
\item[$\boldsymbol\ast_{\mathrm{dt}}$] If $W\in \ell_k$, then a definable subset of $\{g\in \Gamma\sth gW\cap W\in \ell_{k-1}\text{ or }g^{-1}W\cap W\in\ell_{k-1}\}$ is thick in $\Lambda$.  
\end{enumerate}
%\begin{rmk} Obviously, every definable thick Massicot-Wagner system is a definable Massicot-Wagner system. Also, every definable thick Massicot-Wagner system is a thick Massicot-Wagner system and every definable Massicot-Wagner system is a Massicot-Wagner system.\end{rmk}
\begin{rmk} We are interested on the family of definable subsets of $G$. However, note that considering other families $\mathcal{D}$, one can achieve similar notions of $\mathcal{D}$ (thick) Massicot-Wagner systems.\end{rmk}
  
In the definable context, we redefine $\ell^{\mathrm{g}}$ and $\ell^{\mathrm{t}}$ as follows:

\begin{defi} Let $G$ be a definable group and $\Lambda,\Gamma\subseteq G$. 
\begin{enumerate}[label={\rm(\alph*)}, wide]
\item The \emph{largest definable Massicot-Wagner system} $\ell^{\mathrm{dg}}\coloneqq\ell^{\mathrm{dg}}(\Lambda;\Gamma)$ in $\Lambda$ relative to $\Gamma$ is recursively defined (for $W\subseteq G$ definable) by 
\[\begin{array}{ccl} W\in \ell^{\mathrm{dg}}_0&\Leftrightarrow &W\neq \emptyset;\\ W\in \ell^{\mathrm{dg}}_k&\Leftrightarrow &\Lambda\text{ is covered by finitely many translates of a definable }\\ &&\text{ subset of }\{g\in \Gamma\sth gW\cap W\in \ell^{\mathrm{dg}}_{k-1}\text{ or }g^{-1}W\cap W\in\ell^{\mathrm{dg}}_{k-1}\}.\end{array}\]
\item The \emph{largest definable thick Massicot-Wagner system} $\ell^{\mathrm{dt}}\coloneqq\ell^{\mathrm{dt}}(\Lambda;\Gamma)$ in $\Lambda$ relative to $\Gamma$ is recursively defined (for $W\subseteq G$) by
\[\begin{array}{ccl} W\in \ell^{\mathrm{dt}}_0&\Leftrightarrow &W\neq \emptyset;\\ W\in \ell^{\mathrm{dt}}_k&\Leftrightarrow &\text{there is a definable subset of }\\ &&\{g\in \Gamma\sth gW\cap W\in \ell^{\mathrm{dt}}_{k-1}\text{ or }g^{-1}W\cap W\in\ell^{\mathrm{dt}}_{k-1}\}\text{ thick in }\Lambda .\end{array}\]
\end{enumerate}
\end{defi} 

We recall the following result from \cite{hrushovski2022amenability}:
\begin{lem} \label{l:definability massicot system} Suppose $G$ is $0$-definable, $\Lambda$ is $0$-definable and $\Gamma$ is $\bigvee_\emptyset$-definable. Consider the Massicot-Wagner systems $\ell^{\mathrm{g}}\coloneqq \ell^{\mathrm{g}}(\Lambda,\Gamma)$ and $\ell^{\mathrm{t}}\coloneqq\ell^{\mathrm{t}}(\Lambda,\Gamma)$. Let $\varphi(x;\bar{y})$ be a formula without parameters where $x$ is a variable in the sort of $G$. Then,
\[\{\bar{b}\sth \varphi(G;\bar{b})\in\ell^{\mathrm{g}}_k\}\text{ and }\{\bar{b}\sth \varphi(G;\bar{b})\in\ell^{\mathrm{t}}_k\}\text{ are }\bigvee_\emptyset\text{-definable.}\] 
Furthermore, if $G$ and $G\setminus\Lambda$ are semipositively definable, $\Gamma$ is semipositively $\bigvee$-definable and $\varphi(x;\bar{y})$ is a semipositive formula, then 
\[\{\bar{b}\sth \varphi(G;\bar{b})\in\ell^{\mathrm{g}}_k\}\text{ and }\{\bar{b}\sth \varphi(G;\bar{b})\in\ell^{\mathrm{t}}_k\}\text{ are semipositive }\bigvee_\emptyset\text{-definable.}\] 
\end{lem}
This was already noted (with $\Gamma=\Lambda=G$) in \cite[Remark~2.2]{hrushovski2022amenability}. Nevertheless, for the sake of completeness, we provide the proof.
\begin{proof} %Let $\underline{G}(x)$, $\underline{G\setminus\Lambda}(x)$ and $\underline{\Gamma}(x)$ be formulas defining $G$, $G\setminus\Lambda$ and $\Gamma$ respectively. Let $p(x_1,x_2,x_3)$ be the formula defining the group operation $x_1\cdot x_2=x_3$. 
For a formula $\psi(x;\bar{y})$, we write
\[\exists^{k\,\mathrm{gen}}x\, \psi(x;\bar{y})\coloneqq \exists x_1,\ldots,x_k\in G\, \forall z\in \Lambda\, \bigvee^k_{i=1} \psi(x_iz;y),\]
%%\coloneqq \exists x_1\ldots x_k \wedge\bigwedge \underline{G}(x_i)\forall z \vee \underline{G\setminus\Lambda}(z) \exists z_1\ldots z_k\wedge \bigwedge p(x_i,z,z_i)\bigvee^k_{i=1}\psi(x_iz;y).\]
%%Let $q(x_1,x_2,x_3)$ be the formula defining the group operation $x_1^{-1}x_2=x_3$. Note that $q(x_1,x_2,x_3)=\exists x\, x' (p(x,x,x)\wedge p(x,x',x)\wedge p(x',x_2,x_3))$. We get 
\[\exists^{k\,\mathrm{thi}}x\, \psi(x;\bar{y})\coloneqq \forall x_1,\ldots,x_k\in\Lambda\, \bigvee_{i\neq j} \psi(x_i^{-1}x_j;\bar{y}).\]
%%\coloneqq \forall x_1\ldots x_k  \vee \bigvee\underline{G\setminus\Lambda}(x_i) \exists z_{12}\ldots z_{k-1,k} \wedge \bigwedge_{i<j} q(x_i,x_j,z_{ij}) \bigvee_{i<j}\psi(z_{ij};y).\]
For $\bigvee_{<\omega}$-definable sets, we write $\exists^{\mathrm{gen}}x\,\psi(x;\bar{y})=\bigvee_{k\in\N}\exists^{k\,\mathrm{gen}}x\,\psi(x;\bar{y})$ and, similarly, $\exists^{\mathrm{thi}}x\,\psi(x;\bar{y})=\bigvee_{k\in\N}\exists^{k\,\mathrm{thi}}x\,\psi(x;\bar{y})$.

Consequently, $\models\exists^{\mathrm{gen}}x\,\psi(x;\bar{b})$ if and only if finitely many translates of $\psi(G;\bar{b})$ cover $\Lambda$. Similarly, $\models\exists^{\mathrm{thi}}x\,\psi(x;\bar{b})$ if and only if $\psi(G;\bar{b})$ is thick in $\Lambda$.

For $\psi(x;\bar{y})=\bigvee_{i\in I}\psi_i(x;\bar{y})$ with $I$ infinite,  $\exists^{\mathrm{gen}}x\,\psi(x;\bar{y})=\bigvee_{i\in I}\exists^{\mathrm{gen}}x\,\psi_i(x;\bar{y})$ and $\exists^{\mathrm{thi}}x\,\psi(x;\bar{y})=\bigvee_{i\in I}\exists^{\mathrm{thi}}x\,\psi_i(x;\bar{y})$. Consequently, $\models\exists^{\mathrm{gen}}x\,\psi(x;\bar{b})$ if and only if finitely many translates of $\psi_i(G,\bar{b})$ cover $\Lambda$ for some $i\in I$ and $\models\exists^{\mathrm{gen}}x\,\psi(x;\bar{b})$ if and only if $\psi_i(G,\bar{b})$ is thick in $\Lambda$ for some $i\in I$.

For a formula $\psi(x;\bar{y})$, write $\ell^{\mathrm{dg}}_k(\psi(x;\bar{y}))=\{\bar{b}\sth \psi(G;\bar{b})\in\ell^{\mathrm{dg}}_k\}$ and $\ell^{\mathrm{dt}}_k(\psi(x;\bar{y}))=\{\bar{b}\sth \psi(G;\bar{b})\in\ell^{\mathrm{dt}}_k\}$. If $D=\psi(G;\bar{b})$, we get that $\{g\in \Gamma\sth gD\cap D\in\ell^{\mathrm{dg}}_k\text{ or }g^{-1}D\cap D\in\ell^{\mathrm{dg}}_k\}=(\ell^{\mathrm{dg}}_k(\psi(x;\bar{b})\wedge\psi(z^{-1}x;\bar{b}))\vee \ell^{\mathrm{dg}}_k(\psi(x;\bar{b})\wedge\psi(zx;\bar{b})) )\wedge \underline{\Gamma}(z)$ where $\underline{\Gamma}(z)=\bigvee_{j\in J} \theta_i(z)$ defines $\Gamma$. Similarly, $\{g\in \Gamma\sth gD\cap D\in\ell^{\mathrm{dt}}_k\text{ or }g^{-1}D\cap D\in\ell^{\mathrm{dt}}_k\}=(\ell^{\mathrm{dt}}_k(\psi(x;\bar{b})\wedge\psi(z^{-1}x;\bar{b}))\vee \ell^{\mathrm{dt}}_k(\psi(x;\bar{b})\wedge\psi(zx;\bar{b})))\wedge \underline{\Gamma}(z)$.

Hence, we can rewrite the definitions of $\ell^{\mathrm{dg}}$ and $\ell^{\mathrm{dt}}$ as follows:
\begin{enumerate}[align=left, wide]
\item[For $k=0$]
\[\ell^{\mathrm{dg}}_0(\psi(x;\bar{y}))=\ell^{\mathrm{dt}}_0(\psi(x;\bar{y}))=\{\bar{b}\sth \varphi(G;\bar{b})\neq\emptyset\}=\{\bar{b}\sth \exists x\in G\mathrel{} \varphi(x;\bar{b})\}.\]
\item[For $k>0$]
\[\ell^{\mathrm{dg}}_k(\psi(x;\bar{y})){=}\{\bar{b}\sth \exists^{\mathrm{gen}}z ((\ell^{\mathrm{dg}}_{k-1}(\psi(x;\bar{b})\wedge\psi(z^{-1}x;\bar{b})\vee \ell^{\mathrm{dg}}_{k-1}(\psi(x;\bar{b})\wedge\psi(zx;\bar{b}))\wedge\underline{\Gamma}(z))\}\]
\[\ell^{\mathrm{dt}}_k(\psi(x;\bar{y})){=}\{\bar{b}\sth \exists^{\mathrm{thi}}z((\ell^{\mathrm{dt}}_{k-1}(\psi(x;\bar{b})\wedge\psi(z^{-1}x;\bar{b})\vee\ell^{\mathrm{dt}}_{k-1}(\psi(x;\bar{b})\wedge\psi(zx;\bar{b}))\wedge\underline{\Gamma}(z))\}.\]
\end{enumerate}
The conclusion is now evident. 
\end{proof}
By compactness, we get the following consequence:
\begin{coro} If $M$ is $\aleph_0$-saturated and $G$ and $\Lambda$ are definable and $\Gamma$ is $\bigvee$-definable, we get that $\ell^{\mathrm{dg}}$ and $\ell^{\mathrm{dt}}$ are simply the restrictions of $\ell^{\mathrm{g}}$ and $\ell^{\mathrm{t}}$ to the definable subsets.
\end{coro}
We also recall the following fact from \cite{hrushovski2022amenability}:
\begin{lem} \label{l:definable massicot wagner systems from means} Let $G$ be a definable group, $\Lambda\subseteq G$ definable and $(G,\mathcal{A},\mu)$ a mean space invariant by left translations. Let $A,C\in\mathcal{A}$ with $\Lambda A\subseteq C$ satisfy $0<\mu(A)\leq\mu(C)<\infty$. Then, $\ell^\mu\subseteq \bigcap_{k\in\N}\ell^{\mathrm{dt}}_k$.  
\end{lem}
\begin{proof} See \cite[Lemma~2.8]{hrushovski2022amenability}. The idea is that, without loss of generality, we can assume that the structure is $\aleph_0$-saturated. Indeed, we consider the expansion $\lang_{\mu}$ of the language given by adding predicates for $A$ and $C$, a sort to the reals (with its usual ordered field structure) and, for each formula $\varphi(x;\bar{y})$ with $x$ in the sort of $G$, a function symbol $\mu_{\varphi(x;\bar{y})}$ from the sort of $\bar{y}$ to the sort of the reals. We take the $\lang_\mu$-expansion given by $\mu_{\varphi(x;\bar{y})}(\bar{b})=\mu(\varphi(G;\bar{b}))$. We then take an $\aleph_0$-saturated elementary extension $M^*_\mu$. Using the standard part, set $\mu^*(\varphi(G;\bar{b}))=\mathrm{st}\,\mu^*_{\varphi(x;\bar{y})}(\bar{b})$. Take the $\lang$-reduct $M^*$. Now, $M^*$ with $\mu^*$ satisfies the same hypotheses %(i.e. $\Lambda^*A^*\subseteq C^*$, $\mu^*(A^*)=\mu(A)$ and $\mu^*(C^*)=\mu^*(C)$ and $\mu^*(W^*)=\mu(W)>0$ with $W^*\subseteq A^*$) 
and, additionally, is $\aleph_0$-saturated. Since $\ell^{\mathrm{dt}}$ equals $\ell^{\mathrm{t}}$ on definable sets for $\aleph_0$-saturated models, we conclude by \cref{l:massicot wagner systems from means}. 
%Explicitly, write $\ell^{\mathrm{t}^*}$ and $\ell^{\mathrm{dt}^*}$ for the corresponding systems in $M^*$. By \cref{l:massicot wagner systems from means}, we get that $W^*\in\ell^{\mathrm{t}^*}_k$ for all $k$. Now, $\ell^{\mathrm{t}^*}_k=\ell^{\mathrm{dt}^*}_k$ by saturation, so $W^*\in\ell^{\mathrm{dt}^*}_k$ for all $k$. Since $\ell^{\mathrm{dt}}$ is definable by \cref{l:definability massicot system}, we conclude $W\in\ell^{\mathrm{dt}}_k$ in the original elementary substructure $M$. 
\end{proof}

Using definable (thick) Massicot-Wagner systems we get the following model-theoretic versions of \cref{l:basic massicot wagner,t:recursive massicot wagner}:
\begin{lem}[Basic Massicot-Wagner]\label{l:basic massicot wagner model theory} Let $G$ be a definable group in $M$. Let $G\curvearrowright E$ be a group acting on a set. Let $\Lambda\subseteq G$ and $\Gamma\subseteq G$ with $\Gamma$ symmetric containing the identity. Let $(E,\mathcal{F},m)$ be a mean space invariant by the action of $G$. Let $\ell=(\ell_k)_{k\in\N}$ be a definable Massicot-Wagner system, $A\in \bigcap \ell_k$ and $B\in \mathcal{F}$. Assume: 
\begin{enumerate}[label={\rm (\alph*)}, wide]
\item $A'B\in \mathcal{F}$ for any $A'\subseteq A$ which is a finite intersection of translates of $A$. 
\item $0<m(B)\leq m(AB)<\infty$. 
\end{enumerate}
Set $S_\Gamma(AB)\coloneqq\{g\in \langle \Gamma\rangle\sth m(gAB\cap AB)>0\}$. Then, for any $n\in\N$, there is a symmetric definable subset containing the identity $D\subseteq \Gamma\cap AA^{-1}$ such that finitely many translates of $D$ cover $\Lambda$ and $D^n\subseteq S_\Gamma(AB)$. 

Furthermore, suppose that $G$, $A$ and $G\setminus \Lambda$ are semipositively definable and $\Gamma$ is semipositively $\bigvee$-definable. Then, $D$ can be taken semipositively definable too. 
\end{lem}
\begin{proof} By \cref{l:basic massicot wagner}, we find $\tilde{D}=\{g\in \Gamma \sth gW\cap W\in \ell_{k-1}\text{ or }g^{-1}W\cap W\in\ell_{k-1}\}\cup\{1_G\}\subseteq \Gamma\cap AA^{-1}$ with $W\in\ell_k$ such that finitely many translates of $\tilde{D}$ cover $\Lambda$ and $\tilde{D}^n\subseteq S_\Gamma(AB)$. Now, in fact, by hypothesis, $W\in\ell_k$ implies that there is a symmetric definable subset $D\subseteq\tilde{D}$ such that finitely many translates of $D$ cover $\Lambda$ and $D^n\subseteq \tilde{D}^n \subseteq S_\Gamma(AB)$.

Furthermore, without loss of generality, we can take $\ell^{\mathrm{dg}}$ in place of $\ell$. Then, by \cref{l:basic massicot wagner}, we find $\tilde{D}=\{g\in \Gamma\sth gW\cap W\in \ell^{\mathrm{dg}}_{k-1}\text{ or }g^{-1}W\cap W\in \ell^{\mathrm{dg}}_{k-1}\}\cup\{1_G\}\subseteq \Gamma\cap AA^{-1}$ with $W\in\ell^{\mathrm{dg}}_k$ such that finitely many translates of $\tilde{D}$ cover $\Lambda$ and $\tilde{D}^n\subseteq S_\Gamma(AB)$. If $G$, $A$ and $G\setminus \Lambda$ are semipositively definable and $\Gamma$ is semipositively $\bigvee$-definable, we get that $\tilde{D}$ is semipositively definable by \cref{l:definability massicot system} --- note that $W$ is a finite intersections of translates of $A$, so it is semipositively definable too. Hence, we can take $D$ semipositively definable.
\end{proof}

\begin{theo}[Recursive Massicot-Wagner] \label{t:recursive massicot wagner model theory} Let $G$ be a definable group in $M$. Let $G\curvearrowright E$ be a group acting on a set and $\Lambda$ a definable approximate subgroup of $G$. Let $(E,\mathcal{F},m)$ be a mean space invariant by the action of $G$. Let $(\ell_k)_{k\in\N}$ be a definable thick Massicot-Wagner system in $\Lambda$. Let $A\in \bigcap \ell_k$, $B\in \mathcal{F}$ and $n\in\N$. Assume:
\begin{enumerate}[label={\rm (\alph*)}, wide]
\item $A'B\in \mathcal{F}$ for any $A'\subseteq A$ which is a finite intersection of translates of $A$. 
\item $0<m(B)\leq m(AB)<\infty$.
\item $S(AB)\coloneqq\{g\in \langle \Lambda\rangle\sth m(gAB\cap AB)>0\}\subseteq \Lambda^n$.
\end{enumerate}
Then, $\Lambda^n$ has a definable locally compact model (using countably many parameters).
\end{theo}
\begin{proof} Applying \cref{l:basic massicot wagner} recursively as in \cref{t:recursive massicot wagner}, we find a sequence $(D_i)_{i\in\N}$ of commensurable symmetric definable subsets of $\Lambda^n$ containing the identity such that $D^2_{i+1}\subseteq D_i$. Taking $K=\bigcap D_i$ we find a $\bigwedge_{<\omega}$-definable subgroup of $\Gamma=\langle\Lambda\rangle$ contained in $\Lambda^n$ of bounded index in $\langle \Lambda\rangle$. As usual, taking the quotient by $\Gamma^{00}$ with the logic topology (using countably many parameters), we get the definable locally compact model $\faktor{\Gamma}{\Gamma^{00}}$. 
\end{proof}
\begin{proof}[Proof of the model-theoretic remark of \cref{t:recursive massicot wagner mean}]
It is a particular case of \cref{t:recursive massicot wagner model theory}, applying \cref{l:definable massicot wagner systems from means}.
\end{proof}

\bibliographystyle{alphaurl}
\bibliography{A_short_note_on_the_Massicot-Wagner_method.bib}
\end{document}